\documentclass{amsart}
 \usepackage{amsmath}
 \usepackage{amsfonts}
 \usepackage{amssymb}
 \usepackage{amscd}
 \usepackage[all]{xy}
\usepackage[dvipsnames]{xcolor}

\usepackage{pdfsync}
\usepackage{url}
\usepackage{algorithm}
\usepackage{algorithmic}

\usepackage{amssymb,amsmath,amsthm,graphicx,amscd,multind,eufrak,hyperref,}
\theoremstyle{plain}
\newtheorem{thm}{Theorem}

\theoremstyle{definition}
\newtheorem{re}[thm]{Remark}
\newtheorem{question}[thm]{Question}
\newtheorem{example}[thm]{Example}

\newcommand{\cE}{\mathcal{E}}
\newcommand{\cL}{\mathcal{L}}
\newcommand{\cH}{\mathcal{H}}
\newcommand{\cN}{\mathcal{N}}
\newcommand{\cS}{\mathcal{S}}
\newcommand{\cM}{\mathcal{M}}

\newcommand{\cT}{\mathcal{T}}
\theoremstyle{definition}

\begin{document}
\title[The Maximum Likelihood Data Singular Locus]{
\bf The Maximum Likelihood Data Singular Locus}
\author{Emil Horobe\c{t} and Jose Israel Rodriguez}

\subjclass[2010]{}

\maketitle
\begin{abstract}
\vskip -0.5in
For general data, the number of complex solutions to the
likelihood equations is constant and this number is called the (maximum likelihood) ML-degree of the model.
In this article, we describe the special locus of data for which the likelihood equations
have a solution in the model's singular locus.
\end{abstract}

\section{Introduction}
Maximum likelihood estimation is an important problem in statistics.
On a statistical model, one wishes to maximize the likelihood function for given data.
The algebraic approach to this problem determines every critical point of the likelihood function on the the model's closure.
For general data there will be finitely many regular complex critical points.
Moreover,
this number remains constant and is called the maximum likelihood degree of the statistical model.

The (ML) maximum likelihood  degree was introduced in \cite{CHKS06} and \cite{HKS05}.
In \cite{Huh13} Huh relates the ML degree of a smooth model to a topological Euler characteristic and used topological methods to classify varieties with ML degree one \cite{Huh14}.
Recently, Euler characteristics and Gaussian degrees have been used to answer questions about the ML degree of a singular variety \cite{BW15,RW15,Wang15}.

One reason to study the ML degree is because continuous deformations of the data induce continuous deformations of the critical points.
So by deforming generic data to specific data,
 we are able to determine the critical points of the likelihood function as seen in  \cite{HRS12} for example.
For most choices of specific data, the critical points deform to distinct and regular critical points.
However, for special choices of specific data, special behavior may occur.
One type of special behavior is when the deformed critical points are no longer distinct. This was discussed from a computational view in \cite{RT15} for the likelihood equations.

In this paper we discuss a different type of special behavior.
We are interested in deformations of data leading a critical point into the singular locus.
We call the closure of this type of special data the \textbf{(ML) maximum likelihood data singular locus}.

Our main theorem bounds the data singular locus.
We give an algebraic variety contained in the data singular locus and an algebraic variety containing the data singular locus.
These bounds connect dual varieties and Hadamard geometry to the ML data singular locus.
We will give examples to show these bounds are strict.

\begin{thm}\label{ML Main}
Let $X$ be an algebraic statistical model in $\mathbb{P}^{n+1}$.
Then, the following two inclusions hold
\[
(X_{sing}\setminus \cH)*[1:\ldots:1:-1] \subseteq_{(1)} \mathrm{DS}(X)\subseteq_{(2)} (X_{sing}\setminus \cH)*X^*,
\]
where $\mathrm{DS}(X)$ is the data singular locus, $X^*$ is the dual variety, $X_{sing}\setminus \cH$ is the open part of the singular locus where none of the coordinates are zero and the Hadamard product $*$ is considered as in \eqref{Def:Hada}.
\end{thm}

An illustrating example is the following.
\begin{example}\label{EX:first}
Consider the statistical model $\cM$ defined by the vanishing of the polynomial
$f:=p_0^3-p_0p_1^2-p_0p_2^2+2p_1p_2p_3-p_0p_3^2=0$ and $p_i>0$ (the polynomial $f$ is the determinant of a symmetric matrix with $p_0$ along the diagonal).
There is a unique singular point in the model $\cM$ given by $(\frac{1}{4},\frac{1}{4},\frac{1}{4},\frac{1}{4})$.
The algebraic statistical model $X$ we consider is defined by the $equations$:
$$
f=0\text{ and }p_0+p_1+p_2+p_3-p_s=0,
$$
in the unknowns $p_0,p_1,p_2,p_3,p_s$.
For almost all choices of data there are $10$ distinct regular critical points of the likelihood function on $X$, meaning the ML degree is 10.
Complex valued data is not statistically meaningful because the corresponding critical points will also be complex.
But with homotopy continuation, when we deform generic complex valued data to specific real valued data, we also deform the complex critical points to critical points that are statistically meaningful.

For example, for the data $u=(90,2,3,5)$,
we have $10$ real critical points of which  $6$ have positive $p$-coordinates and in the probability simplex.
But if we deform to special data, for example $u=(10,20,30,20)$,  one of the $10$ critical points will go into the singular locus.

In fact, $any$ data $u$ satisfying the algebraic relation
$$ 3u_0^2-2u_0u_1-5u_1^2-2u_0u_2+6u_1u_2-5u_2^2-2u_0u_3+6u_1u_3+6u_2u_3-5u_3^2
$$
will yield a critical point that is contained in the singular locus of $X$.

This is illustrated in the following picture.
We fix $u_3=50$. The data singular locus plotted in the $u_0,u_1,u_2$-coordinates is below.
For data on the curve, we have a critical point going into the singular locus.

\begin{figure}[h]
\begin{center}
\vskip -0.3cm
\includegraphics[scale=0.6]{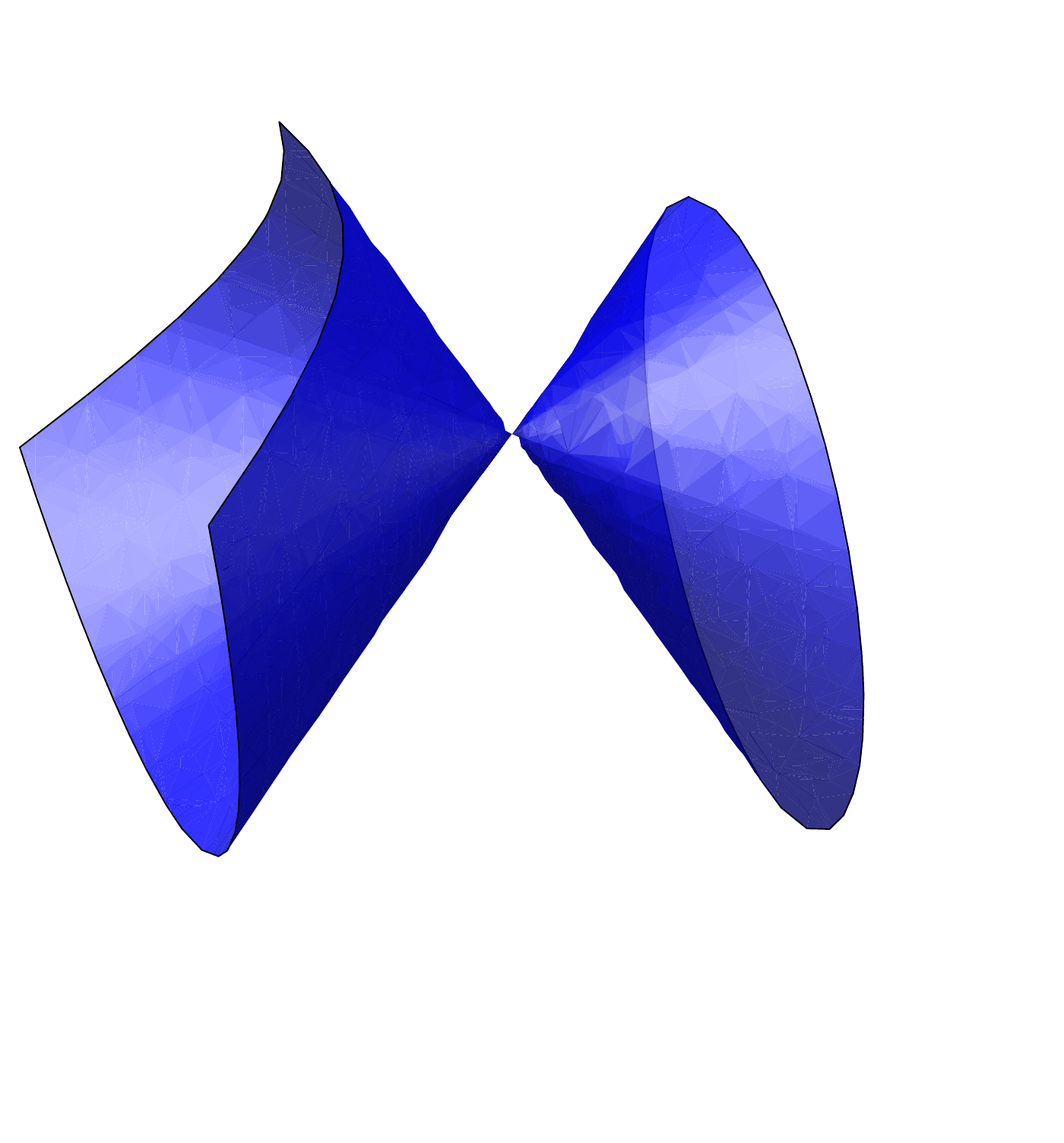}
\vskip -1.6cm
\caption{The surface defined by defined by $-12500 - 100 u_0 + 3 u_0^2 + 300 u_2 - 2 u_0 u_1 - 5 u_1^2 + 300 u_2 - 2 u_0 u_2 +
 6 u_1 u_2 - 5 u_2^2=0$.}
\end{center}
\end{figure}

The central object in this paper is the data singular locus, which is given by the polynomial above.
By convention, we introduce the coordinate $u_s$ and the relation
$(u_0+u_1+u_2+u_3)+u_s=0$.
After doing so, our theorem says that data singular locus contains the point $[\frac{1}{4}:\frac{1}{4}:\frac{1}{4}:\frac{1}{4}:1]*[1:1:1:1:-1]=[\frac{1}{4}:\frac{1}{4}:\frac{1}{4}:\frac{1}{4}:-1]$ and is contained in $[\frac{1}{4}:\frac{1}{4}:\frac{1}{4}:\frac{1}{4}:1]*X^{*}$
(which is exactly the data singular locus for this example).
 \end{example}

The paper is structured as follows.
In  Section~\ref{Sec:Def_ML}, we recall a geometric formulation of the maximum likelihood estimation problem in terms of the (extended) likelihood correspondence \cite{HS14} and dual varieties \cite{Rod14}.
Section~\ref{Sec:DSL}, includes the rigorous definition of the data singular locus and our main result followed by subsection~\ref{Sec:Exa} containing illustrating examples.
In Section~\ref{Sec:Concl}, we conclude by mentioning how our results are related to Hadamard-geometry. Indeed our results can be considered the multiplicative (Hadamard) analog of the recent work of \cite{Hor15} which uses Minkowski sums for the data singular locus related to the Euclidean Distance function and ED-degree \cite{DHOST13}.

\section{The likelihood correspondence and dual varieties}\label{Sec:Def_ML}
An algebraic statistical model $X$ for us is an algebraic variety in $\mathbb{P}^{n+1}$ that is contained in the hyperplane defined by $p_0+p_1+\cdots +p_n-p_s=0$.
The $p_s$ is a homogenization coordinate to keep track of the $s$um of coordinates $p_0,\dots,p_n$.
The radical ideal defining $X$ is denoted by $F_X$.

Our algebraic statistical model is the Zariski closure of a statistical model $\cM$ that's over the real numbers.
A data set $(u_0,\dots,u_n)$ has $u_i$ counting the number of times event $i$ occurs. A probability distribution $(p_0,p_1,\dots,p_n)$ has $p_i$ giving the probability of observing event $i$.
For fixed data, the likelihood function is the monomial $\prod_{i=0}^np_i^{u_i}$ measuring the likelihood of the data $u$ with respect to a probability distribution $p$. The maximum likelihood estimation problem for fixed data $u$ is to maximize the likelihood function on the statistical model $\cM$.

In applied algebraic geometry, we work with Zariski closures in projective space.
These spaces are compact and have a nice geometry.
In algebraic statistics, we take this geometric information to make conclusions about the underlying statistical model.
In this section we will introduce the algebraic approach to maximum likelihood estimation. This approach finds critical points of the homogenized likelihood function $p_0^{u_0}\cdots p_n^{u_n}p_s^{u_s}$ on $X$.
One of these critical points corresponds to the global maximum of the likelihood function on $\cM$ when the maximizer is in the interior of $\cM$.
All the information to determine critical points of the likelihood function is encoded in the so called \textit{extended likelihood correspondence}.
In order to define this object first we have to recall the conormal variety.
For a more detailed reference on computing conormal varieties one can see \cite{RS13}.

For a variety $X$ of $\mathbb{P}^{n+1}$ with coordinates $p_0,\dots,p_n,p_s$,
the {\bf conormal variety} $\cN_X$ in $\mathbb{P}^{n+1}\times\mathbb{P}^{n+1}$ in the coordinates $([p_0:\dots:p_n:p_s],[b_0:\dots:b_n:b_s])$ is defined as the following:
$$
\cN_X:=\overline{\left\{
\left([p,b]\right) : p\in X_{reg} \text{ and } b\perp T_pX
\right\}},
$$
where $X_{reg}$ is the regular locus of $X$ (the collection of nonsingular points).
The projection of the conormal variety to the $b$-coordinates is the {\bf dual variety} $X^{*}$. So we have
$X^*:=\pi_b(\mathcal{N}_X)$
and denote the radical ideal defining $X^*$ by $F_{X^*}$

\begin{example}
In this example, we recall how to compute the conormal variety.
Let $X$ be as in Example \ref{EX:first}.
The Jacobian of $X$ is the matrix of partial derivatives of the defining equations of $X$.
The rows are indexed by equations and columns by unknowns.
The conormal variety is defined using the Jacobian of $X$ augmented by the unknowns $b_0,\dots,b_s$.
Let $F_{minors}$ denote the ideal generated by the $c+1$-minors of the augmented Jacobian below, where $c$ is the codimension of $X$:
\small
$$\left[\begin{array}{ccccc}
b_{0} & b_{1} & b_{2} & b_{3} & b_{s}\\
1 & 1 & 1 & 1 & -1\\
\left(3p_0^{2}-p_1^{2}-p_2^{2}-p_3^{2}\right) & -2\left(p_0p_1-p_2p_3\right) & -2\left(p_0p_2-p_1p_3\right) & 2\left(p_1p_2-p_0p_3\right) & 0
\end{array}\right]
 .$$
\normalsize
Because the singular points of $X$  always cause these minors to vanish, one has to saturate by the singular locus.
We denote the radical ideal defining $X_{sing}$ by $F_{X_{sing}}$.
The ideal of $\cN_X$ is equal to $(F_{minors}+F_X):(F_{X_{sing}})^{\infty}$.
We denote this ideal by $F_\cN$.
The conormal variety has an ideal $F_\cN$ generated by
the generators of $F_X$, the generators of $F_{X^*}$, 1 bilinear polynomial, six degree $\{1,2\}$ polynomials, and nine degree $\{2,1\}$ polynomials.
\end{example}

The \textbf{extended likelihood correspondence} attaches extra information to each point of the conormal variety. These are triples $(p,b,u)$ in $\mathbb{P}_p^{n+1}\times\mathbb{P}_b^{n+1}\times\mathbb{P}_u^{n+1}$~such~that
\[
\cE_X:=\overline{\left\{
\left(p,b,u\right) : (p,b)\in \cN_X \text{ and }
[p_0b_0:\dots:p_nb_s:p_sb_s]=[u_0:\dots:u_n:u_s]
\right\}}.
\]
Note that because $p_0b_0+\cdots+p_nb_n+p_sb_s$ vanishes on $\cN_X$, we must have $u_0+\cdots+u_n+u_s=0$ on $\cE_X$.

There are many natural projections of the extended likelihood correspondence to consider.
The projection $\pi_{p,b}$ maps $\cE_X$ to the conormal variety $\cN_X$.
The projection $\pi_p$, $\pi_b$, and $\pi_u$ project to $X$, $X^*$, and $\mathbb{P}^{n+1}$ respectively.
The degree of the projection of the fibers of $\pi_p$, $\pi_b$, and $\pi_u$
equal the degree of $X$, degree of $X^*$ and ML degree of $X$ respectively.

The map $\pi_{p,u}$ projects to what is called the  likelihood correspondence $\cL_X$ \cite{HS14}.
Typically the likelihood correspondence ignores the the coordinate $p_s$ and $u_s$ because
 $u_s=-(u_0+\cdots+u_n)$ and $p_s=p_0+\cdots+p_n$.
But for cleaner theorems we account for these homogenizing coordinates.

Let $F_{Hada}$ denote the ideal generated by the $2$-minors of
\begin{equation}\label{HadamardMatrix}
\left[\begin{array}{cccc}
p_{0}b_{0} & \dots & p_{n}b_{n} & b_{s}\\
u_{0} & \dots & u_{n} & u_{s}
\end{array}\right].
\end{equation}
These equations represent $[p_0b_0:\dots:p_nb_n:p_sb_s]=[u_0:\dots:u_n:u_s]$.
A subtle issue is when the top row or bottom row of the matrix above is identically zero.
So we let $\cT$ denote the ideal $(p_0b_0,\dots,p_nb_n,p_sb_s)(u_0,\dots,u_n,u_s)$.

The ideal of the extended likelihood correspondence is given by
$(F_\cN+F_{Hada}):(\cT)^\infty$.
Eliminating the $b$-unknowns produces the likelihood correspondence.
Eliminating the $b$ and $u$ unknowns produces the variety $X$.
Eliminating the $p$ and $u$ unknowns produces the variety $X^*$.
Eliminating the $p,b$ produces a $\mathbb{P}^n$ contained in $\mathbb{P}^{n+1}_u$;
the degree of the fiber over a general point of this projection equals the ML degree of $X$; the $p$ coordinates of the fiber over data $u$ correspond to critical points of the likelihood function.

Consider a projective variety $X$ in $\mathbb{P}_p^{n+1}$ and another projective variety $Y$ in $\mathbb{P}_b^{n+1}$.
We define \textbf{the Hadamard product} $*$ \textbf{of $X$ and $Y$}
with respect to the hypersurface $p_0b_0+\cdots+p_nb_n+p_sb_s=0$
to be the image of the product of $X\times Y$ intersected with the hypersurface defined by $p_0b_0+\cdots+p_nb_n+p_sb_s$ to $\mathbb{P}_u^{n+1}$
under the following map:
\begin{equation}\label{Def:Hada}
([p_0:\cdots:p_n:p_s],[b_0:\dots:b_n:b_s])\mapsto
\left[p_{0}b_{0}:\dots:p_{s}b_{s}\right]=\left[u_{0}:\dots:u_{s}\right].
\end{equation}

\begin{example}
Let $X$ be the line given by $[t_0:-2t_0-t_1:-t_0-t_1]$ and $Y$ be the point give by $[a:b:c]$.
Then , $X*Y$ is the set of points in
$[a(t_0):b(-2t_0-t_1):c(-t_0-t_1)]$
satisfying $0=a(t_0)+b(-2t_0-t_1)+c(-t_0-t_1)$.
When $[a:b:c]=[1:2:-3]$ there is a unique point $[0:-2:3]$.
When $[a:b:c]=[1:1:-1]$, we have  $X*Y$ is the line $[t_0:-2t_0-t_1:t_0+t_1]$.
\end{example}

\begin{re}
It is important to distinguish Hadamard products with respect to a hypersurface (or more generally a variety) between the usual Hadamard product (which can be considered as with respect to the entire space).
The reason why we take the Hadamard product with respect to $p_0b_0+\cdots+p_nb_n+p_sb_s$ is because the conormal variety is always contained in the corresponding hypersurface.
\end{re}

\section{Computing the data singular locus and Main result}\label{Sec:DSL}

Our central object of study is the data singular locus.
Geometrically this is determined by first intersecting the extended likelihood correspondence with the set
$\cS:=\pi_p^{-1}(X_{sing}\setminus \cH)$ where $\cH$ is defined by $p_0\dots p_np_s$ and then projecting to data space.
The \textbf{data singular locus} is denoted by
$\mathrm{DS}(X)$ (abbreviating "data singular") and in symbols we have the data singular locus is
\begin{equation}\label{Def:DSL}
\mathrm{DS}(X):=\pi_u(\cS\cap \cE_X).
\end{equation}
One can compute the ML data singular locus in the following way.
Let $F_\cN$ denote the ideal defining the conormal variety $\cN_X$.
Recall $F_{Hada}$ denotes the ideal generated by the $2\times 2$ minors of \eqref{HadamardMatrix}.
Then,
the polynomials of the ideal $F_E:=(F_\cN+F_{Hada})$ witness $\cE_X$; by witness we mean the these equations define $\cE_X$ after saturating by extraneous components. These extraneous components are contained in the coordinate hyperplanes.

Let $F_{X_{sing}}$ denote the ideal of $X_{sing}$.
Then, the data singular locus is the eliminant of $(F_E+F_{X_{sing}}):(p_0\cdots p_np_s\cdot b_0\cdots b_nb_s\cdot u_0\cdots u_nu_s)^{\infty}$ in the $u$ coordinates.

\begin{example}
We summarize the discussion above with the following code written in Macaulay2 \cite{M2}.
\begin{changemargin}{0.5cm}{0.5cm}
\begin{verbatim}
-p are the primal unknowns corresponding to probabilities.
--b are the dual unknowns.
--u are the data
n=4;
R=QQ[p_0..p_(n-1),p_s]**QQ[b_0..b_(n-1),b_s]**QQ[u_0..u_(n-1),u_s];
pList=join(toList(p_0..p_(n-1)),{p_s});
bList=join(toList(b_0..b_(n-1)),{b_s});
uList=join(toList(u_0..u_(n-1)),{u_s});

--f1=f2=0 defines our algebraic statistical model.
f=det matrix{{p_0,p_1,p_3},{p_1,p_0,p_2},{p_3,p_2,p_0}};
f1=f;
f2=sum pList-2*p_s;
I=ideal(f1,f2);

---Determine the Jacobian of f1=f2=0 to get the singular locus.
Jac= jacobian gens I;
jacP=transpose(submatrix(Jac,{0..n},{0..numgens(I)-1}));

--There is only one component in the singular locus outside of the
--coordinate hyperplanes. We denote this component by Sing.
singLocus=decompose (minors(2,jacP)+I)
decSingLocus=for i in singLocus list saturate(i+I,product pList)
Sing=decSingLocus_0

--Sing=!=ideal 1_R
--To compute the conormal variety we consider an auxillary Jacobian
--with dual variables in the top row.
--The conormal variety has points where the dual variables are in
--the row span of the Jacobian. FN defines the conormal variety.
auxJac=matrix{bList}||jacP;
conormal1=minors(3,auxJac)+I;
decConormal1=decompose conormal1;
FN=(flatten for i in decConormal1 list if saturate(i,Sing)==i then i
else {})_0;

--EX defines the extended likelihood correspondence and possibly
--some external factors.
EX=FN+minors(2,matrix{uList, for i to #pList-1 list pList_i*bList_i});

--We need to intersect the extended likelihood correspondence with
--the singular locus.
decUpstairsDSLandExtra=decompose(Sing+EX);
for i in decUpstairsDSLandExtra list
saturate(i,product pList*product bList);
upStairsDSL=(flatten for i in decUpstairsDSLandExtra list
if i== saturate(i,product pList*product bList) then i else {})_0;

--We have to project down to the u-space to get the data singular
--locus.
DSL=eliminate(pList|bList,upStairsDSL)

--The dual variety is the b projection of the conormal
dualVariety=eliminate(pList,FN);

--The Hadamard product of the singular locus and the dual
decHadamardProductUpstairs=decompose ((Sing+dualVariety+
  ideal(sum for i to #pList-1 list pList_i*bList_i))+
  minors(2,matrix (for i to #pList-1 list {pList_i*bList_i,uList_i})))
for i in decHadamardProductUpstairs list eliminate(pList|bList,
  saturate(i,product pList*product bList) )
SingHadaDual=oo_(-1);
\end{verbatim}
\end{changemargin}

\end{example}
\newpage
The following theorem is the main result of this article.
\begin{thm}
Let $X$ be an algebraic statistical model in $\mathbb{P}^{n+1}$. Then, the following two inclusions hold
\[
(X_{sing}\setminus \cH)*[1:\ldots:1:-1] \subseteq_{(1)} \mathrm{DS}(X)\subseteq_{(2)} (X_{sing}\setminus \cH)*X^*.
\]
\end{thm}

\begin{proof}
To prove the first inclusion take any point $p\in X_{sing}\setminus \cH$, then there is a sequence of regular points of the variety $p^{(k)}\in X_{reg}$, such that $p^{(k)}\to p$ with respect to the Euclidean topology.
Then the pairs $(p^{(k)},[1:\cdot:1:-1])$ are always elements of $\mathcal{N}_X$, since the point $[1:\ldots:1:-1]$ is normal to any point of the variety (since $X$ is inside the hyperplane defined by $p_0+...+p_n-p_s$, whose normal vector is $[1:\cdot:1:-1]$).
But then because $\mathcal{N}_X$ is Zariski closed (so closed w.r.t. the Euclidean topology as well) the pair $(p,[1:\cdot:1:-1])$ is also an element of the conormal variety.
But now all the $2\times 2$ minors defining $F_{Hada}$ vanish on the triple $(p,[1:\ldots:1:-1],p*[1:\ldots:1:-1])$, so it is an element of $\mathcal{E}_X$, with $p\in X_{sing}\setminus \cH$.
This means that $p*[1:\ldots:1:-1]$ is in the data singular locus.

To prove the second inclusion, take a point $u\in\mathrm{DS}(X)$, then there exist sequences $p^{(k)}\in X^{reg}\setminus \cH$, $u^{(k)}\in \mathbb{P}^{n+1}$ and $b^{(k)}\in X^*$, such that
\[
(p^{(k)},b^{(k)},u^{(k)})\to (p,b,u)\in \mathcal{E}_X,
\]
for some $p\in X_{sing}\setminus \cH$ and some $b\in \mathbb{P}^{n+1}$. Here the convergence is w.r.t. the Euclidean topology.
At this moment we don't know if $b$ is an element of the dual, because of the closure in $\mathcal{E}_X$. We want to prove that $b\in X^*$.

Since $(p^{(k)},b^{(k)},u^{(k)})$ is in the part of $\mathcal{E}_X$ where $p^{(k)}\in X^{reg}\setminus \cH$ (hence $b^{(k)}\in X^*$ and none of the coordinates of $p^{(k)}$ is zero), then by the definition of $\mathcal{E}_X$ we have that the sequence
\begin{equation}\label{Convergence}
b^{(k)}=[u^{(k)}_0/p^{(k)}_0:\cdots:u^{(k)}_n/p^{(k)}_n:u^{(k)}_s/p^{(k)}_s]=:u^{(k)}/p^{(k)} \in X^*
\end{equation}
is well defined and it converges inside $X^*$, since $X^*$ is Zariski closed (hence closed w.r.t. the Euclidean topology as well).
So when passing to the limit we get that $b=u/p$ for some $b\in X^*$ and $p\in X_{sing}\setminus \cH$, so $u=p*b\in (X_{sing}\setminus \cH)* X^*$.
\end{proof}
We have already seen examples of this theorem with Example~\ref{EX:first}.
We include examples below to show different possibilities of strict containments for $(1)$ and~$(2)$.

\subsection{Examples}\label{Sec:Exa}
In this section we present several useful examples concerning the ML data singular locus. In the first example presented the reader can see that inclusion $(1)$ can be strict and $(2)$ can be an equality.
\begin{example}\label{EX:Ternary}
Let $X$ be given by
$p_2(p_1-p_2)^2+(p_0-p_2)^3=p_0+p_1+p_2-p_s=0$.
The data singular locus is determined by
$2u_0-u_1-u_2=0$.
The bounds according to our theorem are $[\frac{1}{3}:\frac{1}{3}:\frac{1}{3}:-1]$ and
$2u_0-u_1-u_2=u_0+u_1+u_2+u_s=0$.
In this ternary cubic example we see inclusion $(1)$ is strict and $(2)$ is an equality.
\end{example}

In the next example we will see that both inclusions can be strict.
\begin{example}\label{EX:Whitneys}
Here we consider the algebraic statistical model defined by
$$(p_0-p_1)^2p_3-(p_1-p_2)^2(p_2-p_3)=0.$$ After a change of coordinates and dehomogenization this is the Whitney umbrella..

The primary decomposition of $X_{sing}\setminus \cH$ has $2$ components.
The first component is defined by
$ (3p_2 + p_3 - p_s, 3p_1 + p_3 - p_s, 3p_0 + p_3 - p_s)$ and the second component is an embedded point
$[\frac{1}{4}:\frac{1}{4}:\frac{1}{4}:\frac{1}{4}:1]$.
The data singular locus is a product of two polynomials:
$$
\begin{array}{c}
(4u_0^3-3u_0u_1^2+u_1^3-6u_0u_1u_2+3u_1^2u_2-3u_0u_2^2+3u_1u_2^2+u_2^3-15u_0^2u_3\\
+6u_0u_1u_3
-6u_1^2u_3+
24u_0u_2u_3+6u_1u_2u_3-15u_2^2u_3)\cdot(u_0+u_1-u_2-u_3).
\end{array}
$$
Clearly the singular locus is part of the data singular locus, but not equal to it. On the other side inclusion $(2)$ is also strict, because $(X_{sing}\setminus \cH)*X^{*}$ is the whole space, while the data singular locus is of codimension $2$.
\end{example}

In the previous example, we see that $(X_{sing}\setminus \cH)* X^*$ has the expected dimension.
In general by simple dimension count we get that
\[
\mathrm{edim}((X_{sing}\setminus \cH)* X^*)=\mathrm{dim}(X_{sing}\setminus \cH)+\mathrm{dim}(X^*)-1.
\]
In Example~\ref{EX:first} and Example~\ref{EX:Ternary} (where the second inclusion was an equality) the dimension of $(X_{sing}\setminus \cH)* X^*$ was less then expected.

\begin{re}\label{Rem:Smooth}
If the open part of the variety $X\setminus \cH$ is smooth, then as a consequence of Theorem~\ref{ML Main}, we immediately have that both $(1)$ and $(2)$ are equalities..
\end{re}
Reading through this section the reader could see examples of varieties where: both inclusion were strict, both inclusions were equalities and that inclusion $(1)$ can be strict while $(2)$ is an equality (attaining or not the expected dimension). What the reader could not see is an example where $(1)$ is an equality, while $(2)$ is strict. This raises the following question.
\begin{question}\label{Question}
Is there an algebraic statistical model $X$ for which inclusion (1) is an equality, while inclusion (2) is strict?
\end{question}

Throughout the manuscript we have considered only components of $X_{sing}$ that were not contained in the coordinate hyperplanes.
The reason for making this restriction is the following. First, from the statistics perspective, if $p_i=0$ for some $i$, then the likelihood function vanishes and cannot have such a point be a maximizer.
Second, from the mathematics perspective we see in the proof that to have convergence of \eqref{Convergence} we must have nonzero $p$-coordinates.
The next example illustrates the failures that occur when we try to consider components in the singular locus and coordinate hyperplanes.

\begin{example}\label{Ex:Wrong}
For this example we consider the algebraic statistical model defined by the cuspidal cubic $p_0^3+p_1^2p_2=0$. The singular locus of the model is given by the point $[0:0:1:-1]$, which lies in $\cH$. So by our definition the data singular locus is empty. Nevertheless, if we try to determine the locus of those data points where one of the critical points is singular we get a variety with three components
\[
V(u_0 + u_1 + u_2 + u_s)\cup V(u_2 + u_s, u_1, u_0)\cup V(u_1 - 2u_2 - 2u_s, u_0 + 3u_2 + 3u_s).
\]
It is true that $X_{sing}*[1:1:1:-1]$ is a subvariety of this object, and it turns out that $X_{sing}*X^*=X_{sing}$. So this locus is not sandwiched between $X_{sing}$ and $X_{sing}*X^*$, but rather contains data points $u$ for which the corresponding $b$-s, in the extended likelihood correspondence, do not necessarily lie in the dual variety.
\end{example}

\section{Conclusion}\label{Sec:Concl}
We have studied the maximum likelihood data singular locus in this paper, the locus of  data for which one of the critical points to the maximum likelihood equation becomes a singular point.
This is a subvariety of the ML-discriminant of the Lagrange likelihood equations, which was studied from a computational point of view in \cite{RT15}. The study of the ML-discriminant (or it's subvarieties) is an important step towards understanding the MLE, because the ML-degree itself does not characterize the number of statistically relevant critical points. We have to use the ML-discriminant to divide the data space according to the number of meaningful critical points.

Our definition \ref{Def:DSL} of this locus excludes those points which have a singular critical point on the coordinate hyperplane $\cH$, because of the behavior seen in Example~\ref{Ex:Wrong}. 
So our methods do not cover this situation. 
If the data contains zeros then one constructs the so called ML-table, as discussed in \cite{GR13}, which classifies the critical points according to the number of zero coordinates.

Our main result provided upper bounds and lower bounds to the data singular locus and we gave examples showing these bounds are tight (Example~\ref{EX:first}, Example~\ref{EX:Ternary}, and Remark~\ref{Rem:Smooth}), but also raising Question~\ref{Question}.
Furthermore, the maximum likelihood data singular locus is the multiplicative version of the Euclidean data singular locus \cite{Hor15}, where a similar behavior is visible.
In the Euclidean data singular locus the Minkowski sum operation is used instead of the Hadamard product. This is a great analogy between the two notions of algebraic degree.

For future directions, one can study Hadamard products with respect to other hypersurfaces or varieties.
Hadamard products have already been studied in \cite{CMS10,CTY10} in algebraic statistics. Recently, Hadamard products of linear spaces was also studied in \cite{BCK15}.
Maximum likelihood estimation and the study of the data singular locus are additional interesting cases of Hadamard geometry.

\section{Acknowledgements}
The authors would like to thank the organizers of GOAL, Bernd Sturmfels  and Jean-Charles Faugere. The first conversations for this collaboration began at this workshop.
The authors would also like to thank Jan Draisma for his support and advice on this collaboration as well.  The first author was supported
by the NWO Free Competition grant \textit{Tensors of bounded rank},
and the second author was supported by the National Science Foundation under Award No. DMS-1402545.

\bibliographystyle{acm}
\bibliography{refs}

\end{document}